\documentclass[12pt,leqno]{article}

\usepackage[lite]{amsrefs}
\usepackage{latexsym,enumerate}

\usepackage{amssymb, amsfonts, eucal,amsmath,amsthm,upgreek}
\newtheorem{theorem}{Theorem}[section]

\newtheorem{lemma}[theorem]{Lemma}

\newtheorem{remark}[theorem]{Remark}

\newtheorem{corollary}[theorem]{Corollary}

\newcommand{\sect}[1]{\section{#1} \setcounter{equation}{0} }

\newcounter{ca}
\newcommand{\norm}[2]{\left\|#1\right\|_{#2}}

\newcommand{\ds}{\displaystyle}

\newcommand{\Pn}{\mathbb P_n}

 \newcommand{\ec}{\end{comment}}
\newcommand{\bc}{ \begin{comment}
 }

\newcommand{\NN}{{\mathcal N}}

\newcommand{\andd}{\quad\mbox{\rm and}\quad}

\renewcommand{\P}{{\mathcal P}}
\newcommand{\LL}{{\mathcal L}}
\newcommand{\omx}{\lambda_{x}}

\newcommand\e{{\varepsilon}}

\newcommand\w{{\omega}}

 \newcommand\uw{\mathrm{w}}

\newcommand\lw{\Omega^L}
\newcommand\rw{\Omega^R}

\def\be  {\begin{equation}}
\def\ee  {\end{equation}}
\def\ba  {\begin{eqnarray}}
\def\ea  {\end{eqnarray}}
\def\baa {\begin{eqnarray*}}
\def\eaa {\end{eqnarray*}}
\newenvironment{comment}[2]
{\bgroup\vspace{7pt}
\begin{tabular}{|p{5in}|}
\hline \qquad \bf \footnotesize Comment -- to be deleted in the final version \\
\hline
\quad\sl\footnotesize #1#2} {\\ \hline \end{tabular}
\vspace{7pt}\indent\egroup}

\def\updots{\mathinner{\mkern
1mu\raise 1pt \hbox{.}\mkern 2mu \mkern 2mu \raise
4pt\hbox{.}\mkern 1mu \raise 7pt\vbox {\kern 7 pt\hbox{.}}} }

\newcommand{\C}{C}

\newcommand{\R}{\mathbb R}

\newcommand{\N}{\mathbb N}

\newcommand{\ineq}[1]{{\rm(\ref{#1})}}

\newcommand{\ie}{{\em i.e., }}
\newcommand{\eg}{{\em e.g., }}

\newcommand{\bpic}{
\begin{center}
}

\newcommand{\epic}{
\endpspicture
\end{center}
}

\renewcommand{\L}{\mathbb L}

\newcommand{\Poly}{\Pi}

\newcommand{\mon}{\Delta^{(1)}}
\newcommand{\con}{\Delta^{(2)}}

\newcommand{\thm}[1]{Theorem~\ref{#1}}
\newcommand{\lem}[1]{Lemma~\ref{#1}}
\newcommand{\cor}[1]{Corollary~\ref{#1}}

\newcommand{\sectio}[1]{Section~\ref{#1}}
\newcommand{\rem}[1]{Remark~\ref{#1}}

\title{{\sc Interpolatory estimates for convex piecewise polynomial approximation}\thanks{{\it AMS classification:} 41A10, 41A25. {\it Keywords
and phrases:} Convex approximation by polynomials, Degree of approximation, Jackson-type interpolatory estimates.}}

\author{K. A. Kopotun,\thanks
{Department of Mathematics, University of Manitoba,Winnipeg, Manitoba, R3T 2N2, Canada ({\tt
kopotunk@cc.umanitoba.ca}). Supported by NSERC of Canada.}\ \
D. Leviatan\thanks{Raymond and Beverly Sackler School of Mathematical
Sciences, Tel Aviv University, Tel Aviv 6139001, Israel ({\tt
leviatan@tauex.tau.ac.il}).}\ \
and \ I. A. Shevchuk\thanks
{Faculty of Mechanics and Mathematics, Taras
Shevchenko National University of Kyiv, 01601 Kyiv, Ukraine ({\tt
shevchuk@univ.kiev.ua}).}
}

\begin{document}

\maketitle

\abstract{
In this paper, among other things, we show that,
given $r\in\N$, there is a constant $c=c(r)$ such that if   $f\in \C^r[-1,1]$ is convex, then there is a number $\NN=\NN(f,r)$, depending on $f$ and $r$, such that for $n\ge\NN$, there are convex piecewise polynomials $S$ of order $r+2$ with knots at the Chebyshev partition, satisfying
\[
|f(x)-S(x)|\le c(r)\left( \min\left\{ 1-x^2, n^{-1}\sqrt{1-x^2} \right\}  \right)^r \omega_2\left(f^{(r)}, n^{-1}\sqrt{1-x^2} \right),
\]
for all $x\in [-1,1]$.
Moreover, $\NN$ cannot be made independent of $f$.

}

\sect{Introduction, motivation and history}
For $r\in\mathbb N$, let $\C^r[a,b]$, $-1\le a<b\le1$, denote the space of $r$ times continuously differentiable functions on $[a,b]$, and set $\C^0[a,b]:=\C[a,b]$, the space of continuous functions on $[a,b]$, equipped with the uniform norm $\|\cdot\|_{[a,b]}$. Let $\Pn$ be the space of algebraic polynomials of degree $\le n$.

For $f\in \C[a,b]$ and any $k\in\mathbb{N}$, set
\[
\Delta^k_u(f,x;[a,b]):=\begin{cases}
  \sum_{i=0}^k(-1)^i\binom ki f(x+(k/2-i)u),&\quad x\pm (k/2)u\in[a,b],\\
0,&\quad{\rm otherwise},
\end{cases}
\]
and denote by
\[
\w_k(f,t;[a,b]):=\sup_{0<u\le t}\|\Delta^k_u(f,\cdot;[a,b])\|_{[a,b]},
\]
its $k$th modulus of smoothness.
When dealing with $[a,b]=[-1,1]$, we suppress referring to the interval, that is, we denote  $\|\cdot\|:=\|\cdot\|_{[-1,1]}$,
 $\w_k(f,t):=\w_k(f,t;[-1,1])$, etc.

%Let $X_n:=\{x_{j,n}\}_{j=0}^n$, $x_{j,n}=\cos j\pi/n$, $0\le j\le n$, be the Chebyshev partition of $[-1,1]$ (see, \eg \cite{KLS}), and set
%$x_{n+1,n}:=-1$, $x_{-1,n}:=1$.

Finally, let
\be\label{varphi}
\varphi(x)=\sqrt{1-x^2}\quad\text{and}\quad\rho_n(x):=\frac{\varphi(x)}n+\frac1{n^2},
\ee
and note that $\rho_n(x) \sim \varphi(x)/n$, for $x\in [-1+n^{-2}, 1-n^{-2}]$ (we will often use this fact without further discussions).

%In the next paper we are intending to prove the following
%\begin{theorem}\label{thm1} Given $r\in\mathbb{N}$, there is a constant $c=c(r)$ such that if convex function $f\in C^r$, then there is a number $N=N(f,r)$, depending on $f$ and $r$, %such that for $n\ge N$, there are convex polynomials $P_n\in \mathbb P_n$, on $[-1,1]$, satisfying
%\be\label{interpol}
%|f(x)-P_n(x)|\le c(r)\left(\frac{\varphi(x)}n\right)^r\omega_2\left(f^{(r)},\frac{\varphi(x)}n\right),\quad x\in[-1,1].
%\ee
%\end{theorem}
%\begin{remark}\label{rem} For $r=0$ Theorem {\rm\ref{thm1}} was proved by Leviatan \cite{L}, already in 1986, moreover, with $N=1$, thus independent of $f$. The question is arising, %whether \ineq{interpol} may be valid for an $r>0$, even with $\omega_1(f^{(r)},\cdot)$, at least with non-interpolation at the end points, that is, $\varphi(x)$ replaced by %$\rho_n(x)$. (We recall that non-interpolatory pointwise estimates for convex approximation were later given in \cite{SM}.) However, for $r>0$ Theorem {\rm\ref{thm1}} cannot be had %with $N$ independent of the function $f$, as was demonstrated in \cite{P}*. Thus, Theorem {\rm\ref{thm1}} is best possible.
%\end{remark}
%Let $x_{j,n}=\cos j\pi/n$ be the Chebyshev knots (see details at the end of Section \ref{sec3}). On the way to prove Theorem \ref{thm1}, first we obtain the following result on %pointwise convex piecewise polynomial approximation which is interesting in its own right.

Pointwise estimates have mostly been investigated for polynomial approximation of continuous functions in $[-1,1]$ and involved usually the quantity $\rho_n(x)$. The first to deal with such estimates was Nikolskii, and he was followed by Timan, Dzyadyk, Freud and Brudnyi. Detailed discussion may be found in the survey paper \cite{KLPS}, where an extensive list of references is given. Discussion and references to estimates on pointwise monotone and pointwise convex polynomial approximation involving $\rho_n(x)$   may also be found there. Pointwise estimates of polynomial approximation involving $\varphi(x)$ are due originally to Teljakovski\v i and Gopengauz, see \cites{K, GLSW} for extensions and many references. Note that for the latter estimates the approximating polynomials must interpolate the function at the endpoints of the interval. We call such estimates interpolatory.

Throughout this paper, we reserve the notation ``$c$'' for positive constants that  are either absolute or
  may   depend   on the parameters $k$ (the order of the modulus of smoothness) and/or $r$  (the order of the derivative).  We use the notation ``$C$'' and ``$C_i$'', $i\in\N_0$,
  for all other positive constants. We indicate in parentheses the parameters that the constants may depend on. All constants $c$ and $C$ may be different even if they appear in the same line,  but the indexed constants $C_i$ are fixed.

The following theorem is an immediate consequence of \cite{K}*{Corollary 2-3.4}.

\begin{theorem} \label{thk-sim}
Let $r\in\N_0$, $k\in\N$ and $f\in \C^r[-1,1]$. Then for any $n\ge \max\{k+r-1, 2r+1\}$, there is a polynomial $P_n \in \Pn$ such that
\be\label{sim1}
|f(x)-P_n(x)| \le c(r,k) \rho_n^r (x)  \w_k(f^{(r)}, \rho_n(x)), \quad x\in [-1+n^{-2}, 1-n^{-2}] ,
\ee
and
\be \label{sim2}
|f(x)-P_n(x)| \le c(r,k) \varphi^{2r}(x)  \w_k(f^{(r)}, \varphi^{2/k}(x) n^{-2(k-1)/k} ),
\ee
for $x\in [-1,-1+n^{-2}] \cup  [1-n^{-2}, 1]$.
Moreover, for any $\gamma \in\R$, the quantity  $\varphi^{2/k}(x) n^{-2(k-1)/k}$ in \ineq{sim2} cannot be replaced by $\varphi^{2\alpha}(x) n^\gamma$ with $\alpha > 1/k$.
\end{theorem}

\begin{remark}
Since $\w_k(g, \lambda t) \le (\lambda +1)^k \w_k(g, t)$, $\lambda >0$, then, for  $k\ge 2$ and $x$ such that $\varphi(x) \le c/n$ (\ie $x$ is near the endpoints of $[-1,1]$),
\[
\w_k(f^{(r)}, \varphi^{2/k}(x) n^{-2(k-1)/k} ) \le c(k) [\varphi(x) n]^{2-k} \w_k(f^{(r)}, \varphi(x)/n) ,
\]
and so estimates \ineq{sim1}-\ineq{sim2} are stronger than
\be \label{simaux}
|f(x)-P_n(x)| \le c(r,k) \left( \varphi(x)/n \right)^r    \w_k(f^{(r)}, \varphi(x)/n ) , \quad x\in [-1,1] ,
\ee
if $k \le r+2$, and it is known that \ineq{simaux} does not hold in general if $k> r+2$ (see, \eg \cite{K}*{p. 68} for more discussions).
\end{remark}

\begin{remark}
Since $\w_{k_2} (g,  t) \le 2^{k_2-k_1}  \w_{k_1}(g, t)$ if $k_2 > k_1$, estimates \ineq{sim1} for ``large'' $k$ imply those for ``small'' ones. However, this is not the case for estimates \ineq{sim2}, and the fact that \thm{thk-sim} is valid with some $k_2 \in \N$ does not imply that it is valid with $k_1 \in \N$ such that $k_1 <k_2$. For example, let $f_0(x):=(1+x)^{r+1/2}$. Then, $\w_k(f_0^{(r)}, t) \sim \min\{1, \sqrt{t}\}$, for all $k\in\N$. Hence, estimate \ineq{sim2}  becomes, for $x$ is ``close'' to the endpoints of $[-1,1]$,
\[
|f_0(x)-P_n(x)| \leq c(r,k) \phi_k(x,n) , \quad \text{where }\;  \phi_k(x,n) := \varphi^{2r+1/k}(x) n^{-1+1/k},
\]
 and
 \[
 \lim_{x \to \pm 1} \frac{ \phi_{k_2} (x,n)}{ \phi_{k_1} (x,n)} = \infty, \quad \text{if }\; k_2 > k_1 ,
 \]
 \ie this estimate for $k_2$ is not stronger than that for $k_1$.

 At the same time, it is also rather well known that the estimates \ineq{sim1} and \ineq{sim2} for $k_1\in\N$ do not imply those for $k_2\in\N$ with $k_2 > k_1$. Hence, estimates in \thm{thk-sim} for different $k$'s do not follow from one another.
\end{remark}

If we approximate monotone functions by monotone polynomials (we call this ``monotone approximation'' and denote by $\mon$ the class of all non-decreasing functions on $[-1,1]$), then the situation is drastically different.

In \cite{KLS1}, we showed that \ineq{sim1} and \ineq{sim2} with $k=2$ are valid for monotone approximation provided that $n$ is sufficiently large depending on the function $f$ that is being approximated.
 Namely, the following theorem was proved in \cite{KLS1}.

\begin{theorem}\label{thm1} Given $r\in\N$, there is a constant $c=c(r)$ with the property that if $f\in \C^r[-1,1]\cap\mon$, then there exists a number $\NN=\NN(f,r)$, depending on $f$ and $r$, such that for every $n\ge \NN$, there is  $P_n\in \Pn \cap \mon$  satisfying \ineq{sim1} and \ineq{sim2} with $k=2$.
\end{theorem}

We note that $\NN$ in the statement of \thm{thm1}, in general, cannot be made independent of $f$.
% and that even \ineq{sim1} with $r=0$ is, in general, invalid for $k>2$.
%
%
It is still an open question if an analog of this theorem is valid for $k\ge 3$. If $r=0$, then the situation is slightly different (we refer interested readers to \cite{KLS1} for a more detailed discussion of this).

The proof of \thm{thm1} was based, in part, on interpolatory estimates for monotone approximation by piecewise polynomials, first obtained by Leviatan and Petrova \cite{LP} and \cite{LP1}.

It is a natural question if similar type of estimates/results are valid for convex approximation (\ie approximation of convex functions by convex polynomials), and the purpose of this manuscript is to begin investigation in this direction.

% We note that it is still an open problem if an analog of \thm{thm1} is valid for $k\ge 3$.

 \sect{Main results} \label{sec2}

Given an interval $[a,b]$, let   $X   =\{x_j\}_{j=0}^n$ denote a partition of   $[a,b]$, \ie $a =:x_0 < x_1 < \dots < x_{n-1} <  x_n :=b$, and for $m\in\N$,  denote by
  $S(X,m)$ the set of   continuous
   piecewise polynomials of order $m$ on the partition $X$, that is, $s\in S(X,m)$ if %$s\in \C[-1,1]$ and
  $s$ is a piecewise polynomial of degree $m-1$ with knots $x_j$, \ie on each interval $[x_{j-1},x_{j}]$, $1\le j\le n$, the function $s$ is an algebraic polynomial of degree $\le m-1$.

By the Chebyshev partition of $[-1,1]$, we mean the partition
$T_n:=\{t_{j}\}_{j=0}^n$, where
\be \label{cheb}
t_j := t_{j,n} := -\cos(j\pi/n) , \quad 0\leq j \leq n.
\ee
We refer to $t_j$'s as ``Chebyshev knots'' and note that $t_j$, $1\le j \le n-1$, are the extremum points of the Chebyshev polynomial of the first kind of degree $n$.
It is also convenient to denote $t_j:= t_{j,n}:=1$ for $j>n$ and $t_j:=t_{j,n}:=-1$  for $j<0$.
(We note  that Chebyshev knots are sometimes numbered from right to left which is equivalent to defining them as
$\tau_j := \cos(j\pi/n)$, $0\le j \le n$, instead of \ineq{cheb}.)

Denoting by   $\con$ the class of all convex functions in $\C[-1,1]$,
our first result is the following theorem.

\begin{theorem}\label{thm2} Given $r\in\mathbb{N}$, there is a constant $c=c(r)$ such that if $f\in \C^r[-1,1]$ is convex, then there is a number $\NN=\NN(f,r)$, depending on $f$ and $r$, such that for $n\ge\NN$, there are convex piecewise polynomials $S$ of order $r+2$ with knots at the Chebyshev partition $T_n$ (\ie $S\in S(T_n, r+2)\cap\con$), satisfying
\be\label{interspline}
|f(x)-S(x)|\le c(r)\left(\frac{\varphi(x)}n\right)^r \w_2\left(f^{(r)},\frac{\varphi(x)}n\right),\quad x\in[-1,1],
\ee
and, moreover, for $x\in [-1,-1+n^{-2}] \cup  [1-n^{-2}, 1]$,
\be\label{interendtwo}
|f(x)-S(x)|\le c(r)\varphi^{2r}(x) \w_2\left(f^{(r)}, \frac {\varphi(x)}n\right)
\ee
and
\be\label{interendone}
|f(x)-S(x)|\le c(r)\varphi^{2r}(x) \w_1 \left(f^{(r)},  \varphi^2(x) \right) .
\ee
\end{theorem}

\begin{remark}
 As in the case of monotone approximation, $\NN$ in the statement of \thm{thm2}, in general, cannot be independent of  $f$ (see \thm{thmneg2}).  It is still an open problem if  \thm{thm2} is valid for $k\ge 3$ with
\ineq{interspline} and \ineq{interendtwo}/\ineq{interendone} replaced by \ineq{sim1} and \ineq{sim2}.
\end{remark}

  It  is   known that an analog of \thm{thm2}  holds for $r=0$ with $N=1$ (and so, in the case $r=0$, we do not have dependence of $\NN$ on $f$). Indeed, the polygonal line, that is, the continuous piecewise linear $S$, interpolating $f$ at the Chebyshev nodes, is convex and yields \ineq{interspline} with $r=0$ (see, \eg a similar construction in \cite{L}).
 Moreover, one can construct a continuous piecewise quadratic polynomial function $S$ interpolating $f$ at the Chebyshev nodes such that $S$ is convex on $[-1,1]$ and the following estimates hold (see \cite{K-convex}):
 \[
 |f(x)-S(x)|\le c  \w_3\left(f ,\rho_n(x) \right),   \quad x\in [-1+n^{-2}, 1-n^{-2}] ,
 \]
and, for $x\in [-1,-1+n^{-2}] \cup  [1-n^{-2}, 1]$, in addition to \ineq{interendtwo} and \ineq{interendone}, we have
\[
|f(x)-S(x)| \le c    \w_3(f , n^{-4/3}\varphi^{2/3}(x)  ).
 \]
This follows from \lem{genlem} below with $r=0$ and $k=3$ taking into account \rem{mainremark}(ii).

%
%\bc
%
%Need a remark with an explanation/reference discussing that `` \ineq{interspline} and \ineq{interend} with $r=0$ are, in general, invalid for $k>3$.'' For polynomial approximation, this is L. Yushchenko's result for the $4$th modulus ``On One Counterexample in Convex Approximation'' published in UMZh in 2000.
%\ec

Below, we   show that \thm{thm2} is a  consequence of a more general result on approximation by convex piecewise polynomials on general partitions, \thm{splinepositive}.
However, we first show that, indeed, $\NN$ above must depend on $f$.

%\bc We don't really prove dependence on $r$, so I think it is better to just mention dependence on $f$ which is what's interesting.
%\ec

We start with the following negative result  that shows that an analog of \thm{thk-sim} cannot hold for convex polynomial approximation if $r>0$.

\begin{theorem} \label{thmneg1}
For any $r\in\N$ and each $n\in\N$, there is a   function $f\in\C^r[-1,1]\cap \con$, such that for every   polynomial $P_n\in\Pn\cap\con$ and any positive on $(-1,1)$ function $\psi$ such that $\lim_{x\to \pm 1} \psi(x)=0$, either
\be \label{glswineqcon}
\limsup_{x\to -1} \frac{|f(x)-P_n(x)|}{\varphi^2(x) \psi(x)} = \infty \quad \mbox{\rm or}\quad
\limsup_{x\to 1} \frac{|f(x)-P_n(x)|}{\varphi^2(x)\psi(x)} = \infty .
\ee
\end{theorem}
\begin{remark} A similar result is known for monotone approximation (see, \eg \cite{KLS1}*{(1.5)}.
\end{remark}
\begin{proof}
The proof is very similar to that of \cite{GLSW}*{Theorem 4} (see also \cite{petr}), but there are slight variations, and so we give it here for completeness.

Given $n\in\N$ and $r\in\N$, we let $\e:=  n^{-2}$ and define
\[
f(x) := \left\{
\begin{array}{ll}
0 , & \mbox{\rm if } \quad -1 \leq x \le 1-\e ,\\
(x-1+\e)^{r+1} , & \mbox{\rm if } \quad 1-\e <  x \le 1 .
\end{array}
\right.
\]
Then $f\in\C^r[-1,1]\cap \con$, and suppose to the contrary that \ineq{glswineqcon} both fail, \ie suppose that there exists a polynomial $P_n\in\Pn\cap\con$ and a constant $A$ such that
\[
|f(x)-P_n(x)| \le A \varphi^2(x) \psi(x) ,
\]
for all $x$ in some small neighborhoods of $-1$ and $1$. This implies that $f(\pm 1) = P_n(\pm 1)$ and $f'(\pm 1) = P_n'(\pm 1)$. Hence, $P_n(-1)=P_n'(-1)=0$, $P_n(1)=\e^{r+1}$ and $P_n'(1)= (r+1) \e^r$. Since $P_n \in \con$, the first derivative $P_n'$ is non-decreasing, and so $\norm{P_n'}{} = P_n'(1) = (r+1) \e^r$. Additionally, since $P_n'$ is non-negative, $P_n$ is non-decreasing, and so
$\norm{P_n}{} = P_n(1) = \e^{r+1}$. Now, Markov's inequality implies that
\[
(r+1) \e^r = \norm{P_n'}{} \le n^2 \norm{P_n}{} = n^2 \e^{r+1} ,
\]
which is a contradiction (recall that we chose $\e= n^{-2}$).
\end{proof}

We also have the following analog of \thm{thmneg1} for piecewise polynomials that shows that $\NN$ in the statement of \thm{thm2} cannot be made independent of $f$.

\begin{theorem} \label{thmneg2}
For any $r, m, n\in\N$, and each partition  $X=\{x_j\}_{j=0}^n$ of   $[-1,1]$,
 there is a   function $f\in\C^r[-1,1]\cap \con$, such that for every    $s\in S(X,m)\cap\con$ and any positive on $(-1,1)$ function $\psi$ such that $\lim_{x\to \pm 1} \psi(x)=0$, either
\be \label{glswineqcon1}
\limsup_{x\to -1} \frac{|f(x)-s(x)|}{\varphi^2(x) \psi(x)} = \infty \quad \mbox{\rm or}\quad
\limsup_{x\to 1} \frac{|f(x)-s(x)|}{\varphi^2(x)\psi(x)} = \infty .
\ee
\end{theorem}

\begin{proof} We follow, word for word, the proof of \thm{thmneg1} except that we  apply Markov's inequality on $[x_{n-1},1]$ to get
\[
(r+1) \e^r = s'(1) = \norm{s'}{[x_{n-1},1]} \le \frac{2(m-1)^2}{1- x_{n-1}} \norm{s}{[x_{n-1},1]}   = \frac{2(m-1)^2}{1- x_{n-1}} \e^{r+1},
\]
and so arrive at a contradiction by choosing $\e$ to be smaller than $\ds \frac{(r+1) (1- x_{n-1})}{ 2(m-1)^2}$.

\end{proof}

We are now ready to state a more general result on approximation by convex piecewise polynomials on general partitions.
It is convenient to use the following notation:
\be \label{not1}
\lw_k (f, x; [a,b]) :=   \min_{1\le m \le k}   \w_m(f , (x-a)^{1/m} (b-a)^{(m-1)/m}; [a,b])
\ee
and
\be \label{not2}
\rw_k (f, x; [a,b]) :=   \min_{1\le m \le k}   \w_m(f, (b-x)^{1/m} (b-a)^{(m-1)/m}; [a,b]) .
\ee

Note that
\be \label{needpr}
2^{1-k} \w_k(f, b-a; [a,b]) \le \lw_k (f, b; [a,b]) =  \rw_k (f, a; [a,b]) \le  \w_k(f, b-a; [a,b]) .
\ee

\begin{theorem}\label{splinepositive} For every $r\in\mathbb N$ there is a constant $c=c(r)$ with the following property. For each convex function $f\in \C^r[a,b]$, there is a number $H>0$, such that for every partition $X   =\{x_j\}_{j=0}^n$ of $[a,b]$, satisfying
\be\label{22}
x_1-a\le H\quad\text{and}\quad b-x_{n-1}\le H
\ee
there is  a convex piecewise polynomial $s\in S(X,r+2)$, such that
\be  \label{end}
 |f(x)-s(x)|   \le c(x-a)^r \, \lw_2 (f^{(r)}, x;[a,x_{1}]),   \quad   x\in [a,x_{1}] ,
\ee
%\be\label{ends}
%|f(x)-s(x)|\le c(b-x)^r\omega_2(f^{(r)},\sqrt{(b-x)(b-x_{n-1})};[x_{n-1},b]),\quad x\in[x_{n-1},b],
%\ee
\be\label{ends}
|f(x)-s(x)|\le c(b-x)^r \, \rw_2(f^{(r)},x;[x_{n-1},b]),\quad x\in[x_{n-1},b],
\ee
and, for each $j=2,\dots,n-1$ and $x\in[x_{j-1},x_{j}]$,
\begin{align}\label{inner}
|f(x)-s(x)|&\le c(x_j-x_{j-1} )^r\w_2(f^{(r)},x_j-x_{j-1};[x_{j-1},x_{j}])\\
&+c(x_{1}-a)^r\w_2(f^{(r)},x_{1}-a;[a,x_{1}])\nonumber\\
&+c(b-x_{n-1})^r\w_2(f^{(r)},b-x_{n-1};[x_{n-1},b]).\nonumber
\end{align}
\end{theorem}

\thm{splinepositive} is proved in  \sectio{sec4} after we discuss some auxiliary results in  \sectio{sec3}, and we now show how it implies \thm{thm2}.

\begin{proof}[Proof of \thm{thm2}] Suppose that \thm{splinepositive} is proved. Then, if we let $X$ be the Chebyshev partition $T_n = \{t_j\}$, where $n \geq \NN := 3/\sqrt{H}$ , then \ineq{22} is satisfied since
\[
 t_1+1=1-t_{n-1}  = 2 \sin^2\left( \frac{\pi}{2n} \right)    \le  \frac{\pi^2}{2n^2} \le \frac{5}{\NN^2}     \le H .
\]
 Now, as is well known and is not difficult to check, $\varphi(x)/n\sim\rho_n(x)\sim t_j - t_{j-1}$, for $x\in[t_{j-1}, t_j]$, $2\le j\le n-1$.
 Hence \ineq{interspline} follows from \ineq{end} through \ineq{inner}, and \ineq{interendtwo} and \ineq{interendone} follow  from \ineq{end} and \ineq{ends}.
\end{proof}

\sect{Auxiliary results}\label{sec3}

%For the following  lemma, we need the well known Marchaud inequality: if $f\in \C[a,b]$ and $1\le \ell <m$, then
%\be \label{mar}
%\w_\ell(f, t; [a,b]) \le c(m) t^\ell \left( \int_t^{b-a} \frac{\w_m(f, u; [a,b])}{u^{\ell+1}} \, du +   \frac{ \norm{f}{[a,b]}}{(b-a)^\ell} \right) , \quad t> 0.
%\ee

\begin{lemma}  \label{genlem}
Let $r\in\N_0$, $k\in \N$, $f\in \C^r[a,b]$, $C_0\ge 1$,
   and let $\P\in \Poly_{k+r-1}$ be any polynomial such that $\P^{(\nu)}(a) = f^{(\nu)}(a)$, $0\le \nu \le r$, and
%\be \label{auxin1}
% \norm{f^{(r)}-p_{k+r-1}^{(r)} }{[a,b]} \le c  \w_k(f^{(r)}, b-a; [a,b]) .
%\ee
\be \label{auxin}
 \norm{f  -\P   }{[a,b]} \le C_0  (b-a)^{r}  \w_k(f^{(r)}, b-a; [a,b]) .
\ee
  Then, for all $x\in [a,b]$ and all $1\le m \le k$, we have
 \be \label{auxcon}
|f(x)-\P(x)| \leq c C_0    (x-a)^r \w_m(f^{(r)} , (x-a)^{1/m} (b-a)^{(m-1)/m}; [a,b]),
\ee
where the constant $c$ depends only on $k$ and $r$.
\end{lemma}

\begin{remark} \label{mainremark}
\begin{itemize}
\item[(i)] In the case $k=1$, such $\P(f) \in\Poly_r$ is unique; it is   the Taylor polynomial for $f$ at $x=a$, and \ineq{auxin} is rather obvious.

\item[(ii)] In the case $r=0$,   $\P(f)\in\Poly_{k-1}$ may be chosen  to be any polynomial of degree $\le k-1$ interpolating $f$ at $k$ points in $[a,b]$ that include $x=a$ and such that the distance between any two of them is bounded below by $\lambda (b-a)$ for some constant $\lambda >0$ (the constant $C_0$   will depend on $\lambda$ in this case).
\end{itemize}
\end{remark}

Note that, using the notation \ineq{not1}, estimate \ineq{auxcon} can be restated in the following equivalent form:
\[
|f(x)-\P(x)| \leq C (x-a)^r  \lw_k (f^{(r)} , x; [a,b]), \quad C = C(k,r, C_0).
\]
It is also clear that an analog of \lem{genlem} holds if $\P$ interpolates $f$ and its derivatives at $x=b$ instead of $a$, \ie if $P\in \Poly_{k+r-1}$ satisfies \ineq{auxin} and $\P^{(\nu)}(b) = f^{(\nu)}(b)$, $0\le \nu \le r$, then
\be \label{right}
|f(x)-\P(x)| \leq C(k,r, C_0) (b-x)^r  \rw_k (f^{(r)} , x; [a,b]), \quad x\in [a,b] .
\ee

\begin{proof}[Proof of \lem{genlem}]
 We assume that $a=0$. It is obvious that we do not lose any generality making this assumption, but it will make some expressions shorter.

Let $x\in (0,b]$ and   $1\le m\le k$   be fixed throughout this proof.
Denote $\omx := x^{1/m} b^{(m-1)/m}$ and  note that $x  \le \omx \le b$.  It is also convenient   to denote
\[
\uw(u):= \w_m(f^{(r)}, u; [a,b]).
\]

Now, let  $\LL  \in \Poly_{m+r-1}$ be such that $\LL^{(\nu)}(0) = f^{(\nu)}(0)$, $0\le \nu \le r-1$, and $\LL^{(r)}\in \Poly_{m-1}$ is any polynomial satisfying Whitney's inequality on $[0,b]$ and interpolating $f^{(r)}$ at $x=0$.
For example, we can define $\LL^{(r)}(x):= \P^*(x)-\P^*(0)+f^{(r)}(0)$, where $\P^*\in\Poly_{m-1}$ is the polynomial of best approximation of $f^{(r)}$ on $[0,b]$.

We first show that the following estimate holds
%\ineq{auxcon} is valid with $\P$ replaced by $\LL$, \ie  more precisely, we show
\be \label{mainl}
|f(x)-\LL(x)| \leq c x^r \uw(\omx) .
\ee
To this end, with  $g := f-\LL$, since $| g^{(r)}(t) | = | g^{(r)}(t)- g^{(r)}(0) | \le \w_1 (g^{(r)}, x; [0,b])$, $0\le t \le x$,  we   have, if $r\ge 1$,
\begin{align} \label{prineq}
|g(x)|   \le \frac{1}{(r-1)!}   \int_0^x (x-t)^{r-1} | g^{(r)}(t) | \, dt \le x^r \w_1 ( g^{(r)}, x; [0,b]).
\end{align}
Clearly, the same estimate also holds for $r=0$. We now note that $\w_m( g^{(r)}, \cdot ; [a,b]) =\uw(\cdot)$ because $\LL^{(r)}\in\Poly_{m-1}$,
and so
 \ineq{mainl} is verified  if $m=1$.

By Whitney's inequality,  $\norm{g^{(r)}  }{[0,b]} \le c \uw(b)$, and thus
 if $m\ge 2$, then
 using (a particular case of) the well known  Marchaud inequality: if $F\in \C(I)$, then
\[ %\label{mar}
\w_1(F, t; I) \le c(m) t  \left( \int_t^{|I|} \frac{\w_m(F, u; I)}{u^{2}} \, du +   \frac{ \norm{F}{I}}{|I| } \right) ,
\]
where $|I|$ denotes the length of the interval $I$,
%using \ineq{mar} with $\ell=1$
 we have from \ineq{prineq}
\begin{align*}
|g(x)| &\le c x^{r+1} \left(   \int_{x}^{b} \frac{\w_m(g^{(r)}, u; [0,b])}{u^2} \, du + \frac{\norm{g^{(r)}}{[0,b]}}{ b } \right) \\
&\le c x^{r+1} \left(   \int_{x}^{b} \frac{\uw(u)}{u^2} \, du + \frac{\uw(b)}{ b } \right)
\le c x^{r+1} \int_{x}^{2b} \frac{\uw(u)}{u^2} \, du .
\end{align*}
Now, since
$u_2^{-m} \uw(u_2) \le 2^m u_1^{-m} \uw(u_1)$, for   $0<u_1 <u_2$, we have
\begin{align*}
\int_{x}^{2b} \frac{\uw(u)}{u^2} \, du & = \left(  \int_{x}^{\omx} + \int_{\omx}^{2b} \right) \frac{\uw(u)}{u^2} \, du \\
& \le
\uw(\omx) \int_{x}^{\infty} u^{-2}\, du  + 2^m \omx^{-m} \uw(\omx) \int_{0}^{2b} u^{m-2} \, du \\
&=
\frac{\uw(\omx)}{x} \left(1 + \frac{2^{2m-1}}{m-1} \right) ,
\end{align*}
and so \ineq{mainl} is  proved.

Observe now that $Q:= \P-\LL\in \Poly_{k+r-1}$ is such that  $Q^{(\nu)}(0)=0$, $0\le \nu \le r$, and so by Markov's inequality and \ineq{auxin}
\begin{align*}
|Q(x)| & \le x^{r+1} \norm{ Q^{(r+1)}}{[0,b]} \le c x^{r+1} b^{-r-1}\norm{ Q}{[0,b]} \\
& \le
c x^{r+1} b^{-r-1} \left( \norm{f-\P}{[0,b]} + \norm{f-\LL}{[0,b]} \right) \\
& \le c C_0  x^{r+1} b^{-1} \uw(b)
  \le c C_0  x^{r+1} b^{m-1} \omx^{-m} \uw(\omx) \\
  &\le c C_0  x^{r}\uw(\omx),
\end{align*}
which, together with \ineq{mainl},  immediately implies \ineq{auxcon}.
\end{proof}

 \begin{corollary}  \label{ktwo}
Let $r\in\N_0$ and $f\in \C^r[a,b]$,
   and let $L \in \Poly_{r+1}$ be the polynomial of degree $\le  r+1$ such that $L^{(\nu)}(a) = f^{(\nu)}(a)$, $0\le \nu \le r$ and $L(b)=f(b)$.
 Then, for all $x\in [a,b]$, we have
 \be \label{auxconboth}
|f(x)-L (x)| \leq c (x-a)^r \lw_2(f^{(r)} , x;  [a,b]) ,
\ee
 %\be \label{auxcontwo}
%|f(x)-L (x)| \leq c (x-a)^r \w_2(f^{(r)} , \sqrt{(x-a)(b-a)};  [a,b])
%\ee
%and
% \be \label{auxconone}
%|f(x)-L (x)| \leq c (x-a)^r \w_1 (f^{(r)} , x-a;  [a,b]) ,
%\ee
where the constant  $c$ depends  only on  $r$.
\end{corollary}

\begin{remark}
We note that the estimate \ineq{auxconboth} with $\w_2(f^{(r)} , \sqrt{(x-a)(b-a)};  [a,b])$ instead of $\lw_2(f^{(r)} , x;  [a,b])$  appeared, among other places, in \cite{LP}*{Corollary 3.5}.
\end{remark}

 It is   clear than an analog of this results holds if interpolation of the derivatives of $f$ takes place at $x=b$ instead of $a$, \ie
if $L \in \Poly_{r+1}$ is the polynomial of degree $\le  r+1$ such that $L^{(\nu)}(b) = f^{(\nu)}(b)$, $0\le \nu \le r$ and $L(a)=f(a)$, then
 \be \label{rightcor}
|f(x)-L (x)| \leq c (b-x)^r  \rw_2  (f^{(r)} , x;  [a,b]) , \quad x\in [a,b].
\ee

\begin{proof}[Proof of \cor{ktwo}]
As in the proof of \lem{genlem}, it is clear that we do not lose generality by assuming $a=0$.
Now, if  $g:= f-L$, then $g(b)=0$ and $g^{(\nu)}(0) =0$, $0\le \nu \le r$, and
by \lem{genlem}, it is sufficient to prove that
$\norm{g  }{[0,b]} \le c   b^r \w_2 (g^{(r)}, b; [0,b])$.

Since
$g(x) = \frac{1}{(r-1)!} x^r  \int_0^1 (1-t)^{r-1} g^{(r)}(xt) dt$, equality
$g(b)=0$ implies that   $\int_0^1 (1-t)^{r-1} g^{(r)}(bt) dt = 0$, and so
\begin{align*}
\norm{g}{[0,b]} & \le c b^r  \sup_{0< t\le 1, \; 0\le x\le b} | g^{(r)}(xt) - x g^{(r)}(bt)/b|  \\
&  \le c b^r \sup_{0< t\le 1, \; 0\le y\le  bt} | g^{(r)}(y) - y g^{(r)}(bt)/(bt)| \\
&\le c b^r  \sup_{0< t\le 1} \w_2( g^{(r)}, bt; [0, bt])
 \le c  b^r \w_2( g^{(r)}, b; [0, b]),
\end{align*}
as needed. Here, the second last estimate follows from Whitney's inequality using the observation that $l(y) = y g^{(r)}(bt)/(bt)$ is the linear polynomial interpolating $g^{(r)}$ at $0$ and $bt$.
\end{proof}

The following lemma was proved in \cite{LS}.

\begin{lemma}[\cite{LS}*{Corollary 2.4}]
\label{lemLS3} Let $k\in\N$ and let $f\in \C^2[a,b]$ be convex.  Then there exists a convex polynomial $P\in\Poly_{k+1}$, satisfying
$P(a)=f(a)$ and $P(b)=f(b)$, and either $P'(a)= f'(a)$ and $P'(b)\le f'(b)$, or $P'(a)\ge f'(a)$ and $P'(b)= f'(b)$, such that
\[
\|f-P\|_{[a,b]}\le c(k) (b-a)^2 \w_k(f'',b-a;[a,b]).
\]
\end{lemma}

%Also, the proof \cite{LS}*{Lemma 2.3} easily leads to the following estimate for $r=1$. Namely,

We also need the following analog of \lem{lemLS3} for $r=1$.

\begin{lemma}\label{lemLS4} Let $f\in \C^1[a,b]$ be convex.  Then there exists a convex polynomial $P\in\Poly_2$ (that is a convex parabola), satisfying
$P(a)=f(a)$ and $P(b)=f(b)$, and either $P'(a)= f'(a)$ and $P'(b)\le f'(b)$, or $P'(a)\ge f'(a)$ and $P'(b)= f'(b)$, such that
\[
\|f-P\|_{[a,b]}\le c(b-a)\w_2(f',b-a;[a,b]).
\]
\end{lemma}

\begin{proof} It is clear that it is sufficient to prove this lemma for $[a,b]=[0,1]$, since we can then apply a   linear transformation to a general interval.

Additionally, by subtracting a linear polynomial interpolating $f$ at the endpoints  from $f$ we can assume, without loss of generality, that $f(0)=f (1)=0$.
We now define
\[
P(x) := \left\{
\begin{array}{ll}
 (x-x^2) f'(0) , & \mbox{\rm if} \; f'(0)+f'(1) \ge 0 , \\
 (x^2-x)f'(1) , & \mbox{\rm otherwise.}
 \end{array}
 \right.
\]
Clearly, $P$ is convex and satisfies the stated interpolation conditions. In fact, it is a Lagrange-Hermite polynomial interpolating $f$ at $0$ and $1$, and $f'$ either at $0$ or at $1$.
Hence, we can   use, for example, \cor{ktwo} with $r=1$   or its analog (see \ineq{rightcor}) to conclude that the needed estimate also holds. Alternatively, we can follow   the proof of \cite{LS}*{Lemma 2.3} to verify this estimate.
\end{proof}

%We prove the lemma for $[a,b]=[0,1]$ and then apply a linear transformation to a general interval. In $[0,1]$, we take the linear $p$ interpolating $f'$ at both endpoints, that is, $p(0)=f(0)$ and $p(1)=f(1)$. This guarantees that $p$ is nondecreasing and, by Whitney's theorem
%$$
%\|f'-p\|\le c\omega_2(f',1).
%$$
%The proof then follows verbatim the proof of \cite{LS}*{Lemma 2.3}.

An immediate consequence of Lemmas \ref{lemLS3} and \ref{lemLS4} is the following result.

\begin{lemma}\label{lemLS1} If $r\in\N$ and $f\in \C^r[a,b]$ is convex on $[a,b]$, then for each partition $X   =\{x_j\}_{j=0}^n$ of $[a,b]$  there is a convex piecewise polynomial $\sigma\in S(X,r+2)$,
satisfying, for each $j=1,\dots,n$,
\be\label{omeg2}
\|f-\sigma\|_{[x_{j-1},x_{j}]}\le c(r) (x_{j}-x_{j-1})^r\w_2(f^{(r)},x_{j}-x_{j-1};[x_{j-1},x_{j}]),
\ee
\be\label{omeg5}
\sigma'(x_{j-1}+)\ge f'(x_{j-1}) ,    \qquad \sigma'(x_{j}-)\le f'(x_{j}),
\ee
and
\be\label{omeg3} \sigma(x_j)=f(x_j).
\ee
\end{lemma}

We will now discuss construction of polynomial pieces  near the endpoints of $[a,b]$.

For $f\in \C^r[a,b]$ and $0<h \le b$, denote by $L_{r,h}(f, x)$ the Lagrange-Hermite polynomial of degree $\le r+1$ such that
\[
L_{r,h}^{(\nu)}(f,a)=f^{(\nu)}(a),\quad 0\le \nu \le r , \andd L_{r,h}(f,a+h)=f(a+h) .
\]
We also denote
\[
\L_{r,h}(f,x):=\int_a^x L_{r-1,h}(f',t)dt+f(a). %\quad x\in[a,a+h].
\]

%\begin{lemma}\label{newlem1} Let $r\in\N_0$ and $h>0$. If $g\in \C^r[a,a+h]$, then
%\[
%|g(x)- L_{r,h}(g,x)|\le c(r) (x-a)^r \w_2(g^{(r)},\sqrt{(x-a)h};[a,a+h]),\quad x\in[a,a+h] .
%\]
%\end{lemma}

\begin{lemma}  \label{newlem1}
Let $r\in\N$ and $h>0$. If $f\in \C^r[a,a+h]$, then
\[
|f(x)- \L_{r,h}(f,x)|\le c(r) (x-a)^r \lw_2(f^{(r)},x;[a,a+h]),\quad x\in[a,a+h] .
\]
\end{lemma}

\begin{proof} It follows from \cor{ktwo} that, for
  $r\in\N $, $h>0$ and $g\in \C^{r-1}[a,a+h]$, we have
\[
|g(x)- L_{r-1,h}(g,x)|\le c(r) (x-a)^{r-1} \lw_2(g^{(r-1)},x;[a,a+h]),\quad x\in[a,a+h] .
\]

For any $x\in[a,a+h]$, we have
\begin{align*}
|f(x)-{\L}_{r,h}(f,x)| &= \left|\int_a^x(f'(t)-L_{r-1,h}(f',t)dt \right|\\
&\le c\int_a^x (t-a)^{r-1}\lw_2(f^{(r)},t ;[a,a+h])dt\\
&\le c(x-a)^r\lw_2(f^{(r)},x;[a,a+h]),
\end{align*}
and the proof is complete.
\end{proof}

It was shown in \cite{LP}*{Lemma 3.1} that, for a nondecreasing $g\in \C^r[a, b]$, $r\in\N$,  the polynomial  $L_{r,h}(g,\cdot)$ is also nondecreasing on $[a, a+h]$ provided that $h<b-a$ is sufficiently small depending on $f$. Note that this statement also is valid (and is trivial) if $r=0$.

\begin{lemma}[\cite{LP}*{Lemma 3.1}]               \label{lem1}
Let $r\in\N_0$ and let $g\in \C^r[a,b]$  be  nondecreasing on $[a,b]$. Then there is a number $H>0$, such that for all
$h\in(0,H)$ the polynomials $L_{r,h}(g,\cdot)$ are   nondecreasing on $[a,a+h]$.
\end{lemma}

%\begin{lemma}\label{lem2}\cite{LP}*{Corollary 3.4} Given $g\in \C^r[0,1]$, $r\ge0$, we have
%\be\label{distance}
%|g(x)-L_{r,1}(g,x)|\le cx^r\omega_2(g^{(r)},\sqrt x;[0,1]),\quad x\in[0,1].
%\ee
%\end{lemma}

%Suppose now that a function $f\in \C^r[a,a+h]$, $r\in\N$, is convex, and define
%\[
%\L_{r,h}(f,x):=\int_a^x L_{r-1,h}(f',t)dt+f(a),\quad x\in[a,a+h].
%\]

A trivial consequence of \lem{lem1} is the following result.

\begin{corollary}\label{lem11} Let $f\in \C^r[a,b]$, be convex on $[a,b]$. Then there is a number $H>0$, such that for all
$h\in(0,H)$ the polynomials ${\L}_{r,h}(f,\cdot)$ are convex on $[a,a+h]$.
\end{corollary}

By considering $\widetilde f (x):= f(a+b-x)$ instead of $f$ we also get analogous statements and interpolatory estimate near the endpoint $b$ instead of $a$.

%It follows by \lem{lem2} that
%\begin{lemma}\label{lem12} Given $f\in \C^r[0,1]$, $r\ge1$,
%\be\label{distance1}
%|f(x)-\L_{r,1}(f,x)|\le cx^r\omega_2(f^{(r)},\sqrt x;[0,1]),\quad x\in[0,1].
%\ee
%\end{lemma}
%\begin{proof} We have
%\begin{align*}
%|f(x)-{\L}_{r,1}(f,x)|&=|\int_0^x(f'(t)-L_{r-1,1}(f',t)dt|\\
%&\le c\int_0^xt^{r-1}\omega_2(f^{(r)},\sqrt t;[0,1])dt\\
%&\le cx^r\omega_2(f^{(r)},\sqrt x;[0,1]).
%\end{align*}
%This completes the proof.
%\end{proof}
%
%
%
%Applying the linear substitution $x=yh$, we obtain
%\begin{corollary}\label{cor2} Let $r\ge 1$ and $h>0$. If $f\in \C^r[0,h]$, then
%\[
%|f(x)-{\L}_{r,h}(f,x)|\le cx^r\omega_2(f^{(r)},\sqrt{xh};[0,h]),\quad x\in[0,h].
%\]
%\end{corollary}
%

%
%This readily implies that
%\begin{corollary}\label{cor3} Let $r\ge 1$ and $h>0$. If $f\in \C^r[1-h,1]$ and $g(x):=f(1-x)$, then
%\[
%|f(x)-{\L}_{r,h}(g,1-x)|\le c(1-x)^r\omega_2(f^{(r)},\sqrt{(1-x)h};[1-h,1]),\quad x\in[1-h,1].
%\]
%\end{corollary}
%Again, a trivial consequence of \lem{lem1} is
%\begin{corollary}\label{lem111} Let $f\in \C^r[0,1]$, $r\ge1$, be convex on $[0,1]$, and let $g(x):=f(1-x)$. Then, there is a number $\widetilde H>0$, such that for all
%$h\in(0,\widetilde H)$ the polynomials $\L_{r,h}(g,1-\cdot)$ are convex on $[1-h,1]$.
%\end{corollary}

Thus,  denoting  $\widetilde\L_{r, h}(f,x):=\L_{r,  h}(g,a+b-x)$, where $g(x):=f(a+b-x)$, we can summarize the above results as follows.

\begin{lemma}\label{lem112}
Let $r\in\N$, and let $f\in \C^r[a,b]$  be convex on $[a,b]$. Then there is a number $H>0$, such that for all
$h, \widetilde h \in(0,H)$  there are  polynomials $\L_{r,h}(f,\cdot)$ and $\widetilde\L_{r,  \widetilde h}(f,\cdot)$ of degree $\le r+1 $, such that
\begin{itemize}
\item[(i)]
 ${\L}_{r,h}(f,\cdot)$ is convex on $[a,a+h]$ and  $\widetilde{\L}_{r, \widetilde h}(f,\cdot)$ is convex on $[b- \widetilde h,b]$,
 \item[(ii)]
$\ds
|f(x)-{\L}_{r,h}(f,x)|\le c(r) (x-a)^r \lw_2(f^{(r)},x;[a,a+h])$, $ x\in[a,a+h]$,
 \item[(iii)]
$\ds
|f(x)-\widetilde{\L}_{r, \widetilde h}(f,x)|\le c(r)(b-x)^r\rw_2(f^{(r)},x;[b- \widetilde h,b])$, $x\in[b- \widetilde h,b]$,
\item[(iv)]
$\ds {\L}_{r,h}'(f,a+h)=f'(a+h)$ and $\widetilde{\L}_{r, \widetilde h}'(f,b- \widetilde h)=f'(b- \widetilde h)$.
\end{itemize}
\end{lemma}

%\begin{lemma}\label{lem112} Let $f\in \C^r[0,1]$, $r\ge1$, be convex on $[0,1]$. Then there is a number $H>0$, such that for all
%$h\in(0,H)$ and $\widetilde h\in(0,H)$ there are  polynomials $\L_{r,h}(f,\cdot)$ and $\widetilde\L_{r,\widetilde h}(f,\cdot)$ of degree $\le r+1 $, such that\\
%\centerline{${\L}_{r,h}(f,\cdot)$ is convex on $[0,h],\qquad$ $\widetilde{\L}_{r,\widetilde h}(f,\cdot)$ is convex on $[1-\widetilde h,0]$,}
%\[
%|f(x)-{\L}_{r,h}(f,x)|\le c_0x^r\omega_2(f^{(r)},\sqrt{xh};[0,h]),\quad x\in[0,h],
%\]
%\[
%|f(x)-\widetilde{\L}_{r,\widetilde h}(f,x)|\le c_0(1-x)^r\omega_2(f^{(r)},\sqrt{(1-x)\widetilde h};[1-\widetilde h,1]),\quad x\in[1-\widetilde h,1],
%\]
%$$
%{\L}_{r,h}'(f,h)=f'(h)\quad\text{and}\quad \widetilde{\L}_{r,\widetilde h}'(f,1-\widetilde h)=f'(1-\widetilde h).
%$$
%\end{lemma}

\sect{Proof of \thm{splinepositive}}\label{sec4}

It   suffices to prove this theorem for $[a,b]=[0,1]$, since we can then get the general result by applying a linear transformation. Additionally, by subtracting a linear polynomial interpolating $f$ at the endpoints we can assume that $f(0)=f(1)=0$. It is also clear that we can assume that $f$ is not a constant function, and so, because of its convexity, $f(x)<0$, for all $x\in (0,1)$. Now, denote $M:= \norm{f}{[0,1]} >0$, and let
 $x_* \in(0,1)$ be such that $f(x_*)=\min_{x\in[0,1]}f(x) = -M$.

Suppose now that a positive number $H_1<\min\{x_*,1-x_*\}$ is so small that
\[
\max\{ -f(H_1), -f(1-H_1)\} < M/2
\]
and
\be\label{small}
4c_0H_1^r\w_2(f^{(r)},H_1;[0,1])< M ,
\ee
where $c_0$ is the maximum of constants $c=c(r)$ from inequalities (ii) and (iii) in  \lem{lem112}.
Now, let $H$ be the number from \lem{lem112}, and without loss of generality, we assume that $H \le H_1$.

Suppose that a partition $X=\{x_j\}_{j=0}^n$ of $[0,1]$ satisfies \ineq{22}, and set $h:=x_{1}$ and $\widetilde h:=1-x_{n-1}$.

We are now ready to construct the  piecewise polynomial $s\in S(X,r+2)$ that yields \thm{splinepositive}.
  First, let
\[
s(x):=\begin{cases}
\L_{r,h}(f,x),& \quad\text{if}\quad x\in[0,x_{1}),\\
\widetilde{\L}_{r,\widetilde h}(f,x),& \quad\text{if}\quad x\in (x_{n-1},1],
\end{cases}
\]
where $\L_{r,h}$ and $\widetilde{\L}_{r,\widetilde h}$ are the   polynomials from \lem{lem112}, and note that estimates (ii) and (iii) of \lem{lem112}  immediately imply \ineq{end} and \ineq{ends}.

Now, suppose that $\sigma\in S(X,r+2)$ is the piecewise polynomial  from  \lem{lemLS1}. Note that we cannot simply define $s$ to be $\sigma$ on $[x_1, x_{n-1}]$ because $s$ would then  be   possibly   discontinuous at $x_1$ and $x_{n-1}$, because polynomials $\L_{r,h}$ and $\widetilde{\L}_{r,\widetilde h}$ do not necessarily interpolate $f$ at $x_1$ and $x_{n-1}$, respectively.

We are now going to show how to overcome this difficulty.

Set
\[
\delta :=\L_{r,h}(f,x_{1})-f(x_{1}),\quad \widetilde \delta :=\widetilde\L_{r,\widetilde h}(f,x_{n-1})-f(x_{n-1}),\quad\text{and}\quad \widehat \delta :=\delta-\widetilde \delta,
\]
and note that by virtue of \ineq{small} and \ineq{needpr} estimates (ii) and (iii) in \lem{lem112} imply that
\[
|\delta|<M/4\quad\text{and}\quad|\widetilde \delta|<M/4,
\]
so that $|\widehat \delta|<M/2$.

Denote by $l$ the tangent line to $f$ at $x_{1}$, and by $\widetilde l$ the tangent line to $f$ at $x_{n-1}$. Then we have
\begin{align}\label{x1}
f(x_{n-1})-l(x_{n-1})&=f(x_{n-1})-f(x_*)+f(x_*)-l(x_*)+l(x_*)-l(x_{n-1})\\
&>f(x_{n-1})-f(x_*)\ge f(1-H_1)-f(x_*)
> M/2 ,\nonumber
\end{align}
and similarly
\begin{align}\label{xn-1}
f(x_{1})-\widetilde l(x_{1})&=f(x_{1})-f(x_*)+f(x_*)-\widetilde l(x_*)+\widetilde l(x_*)-\widetilde l(x_{1})\\
&>f(x_{1})-f(x_*)\ge f(H_1)-f(x_*)
> M/2 .\nonumber
\end{align}

To define $s$ on $[x_{n-1},x_1]$ we consider two cases: $\widehat \delta\ge 0$ and $\widehat \delta<0$.

\medskip

\noindent {\bf Case 1:  $\widehat \delta\ge0$}

\medskip

In this case, we define
\[
s(x):=\lambda(\sigma(x)-l(x))+l(x)+\delta, \quad x\in[x_1,x_{n-1}] ,
\]
where
\[
\lambda:= 1- \frac{ \widehat \delta}{f(x_{n-1})-l(x_{n-1})}.
%
%\lambda:=\frac{f(x_{n-1})-l(x_{n-1})-\widehat \delta}{f(x_{n-1})-l(x_{n-1})}.
\]
It is straightforward to check that
\[
s(x_{1})={\L}_{r,h}(f,x_{1})  \andd s(x_{n-1})=\widetilde{\L}_{r,\widetilde h}(f,x_{n-1}) ,
\]
and so $s$ is continuous on $[0,1]$.

Now,  since by \ineq{x1}, $0<\lambda\le1$, it follows from (iv) of \lem{lem112}
 that
\[
s'(x_{1}+)=\lambda(\sigma'(x_{1}+)-f'(x_{1}))+f'(x_{1})\ge f'(x_{1})=s'(x_{1}-)
\]
and, since $f'$ is nondecreasing,
\[
s'(x_{n-1}-)=\lambda\sigma'(x_{n-1}-)+(1-\lambda)f'(x_{1})\le f'(x_{n-1})=s'(x_{n-1}+),
\]
and so $s$ is a convex function on $[0,1]$.

Since $\sigma-l$ is a nondecreasing nonnegative function on  $[x_{1},x_{n-1}]$, we also have, for $x\in[x_{1},x_{n-1}]$,
\begin{align*}
\sigma(x)-s(x)&=\sigma(x)-\lambda(\sigma(x)-l(x))-l(x)-\delta\\
&=(1-\lambda)(\sigma(x)-l(x))-\delta\ge -\delta
\end{align*}
and
\begin{align*}
\sigma(x)-s(x)&=  %\sigma(x)-\lambda(\sigma(x)-l(x))-l(x)-\delta\\
%&=
(1-\lambda)(\sigma(x)-l(x))-\delta\\
&\le(1-\lambda)(\sigma(x_{n-1})-l(x_{n-1})) -\delta \\
& = \frac{ \widehat \delta}{f(x_{n-1})-l(x_{n-1})} (\sigma(x_{n-1})-l(x_{n-1})) - \delta \\
& =-\widetilde \delta .
\end{align*}

Hence,
\[
|f(x)-s(x)|\le|f(x)-\sigma(x)|+|\delta|+|\widetilde \delta|,\quad  x \in[x_{1},x_{n-1}],
\]
and, together with estimates (ii) and (iii) of \lem{lem112} and  \ineq{omeg2},
this proves \ineq{inner}.

\medskip

\noindent {\bf Case 2:  $\widehat \delta<0$}

\medskip

In this case, we define
\[
s(x):=\widetilde \lambda(\sigma(x)-\widetilde l(x))+\widetilde l(x)+\widetilde \delta, \quad x\in[x_{1},x_{n-1}] ,
\]
where
\[
\widetilde\lambda:= 1+ \frac{\widehat \delta}{f(x_{1})-\widetilde l(x_{1})},
%
% \lambda:=\frac{f(x_{1})-\widetilde l(x_{1})+\widehat \delta}{f(x_{1})-\widetilde l(x_{1})},
\]
and we proceed as above using  \ineq{xn-1} instead of \ineq{x1}. This completes the proof. \qed

\end{document}